\documentclass[11pt,reqno]{amsart}
\usepackage[a4paper,margin=30mm]{geometry}
\usepackage{amsmath,amssymb,mathtools}
\usepackage{microtype}
\usepackage[colorlinks=true,linkcolor=blue,citecolor=blue,urlcolor=blue]{hyperref}

\newtheorem{theorem}{Theorem}[section]
\newtheorem{proposition}[theorem]{Proposition}
\newtheorem{lemma}[theorem]{Lemma}
\newtheorem{corollary}[theorem]{Corollary}

\theoremstyle{definition}
\newtheorem{definition}[theorem]{Definition}

\theoremstyle{remark}
\newtheorem{remark}[theorem]{Remark}

\numberwithin{equation}{section}

\newcommand{\R}{\mathbb{R}}
\newcommand{\C}{\mathbb{C}}
\newcommand{\Q}{\mathbb{Q}}
\newcommand{\Z}{\mathbb{Z}}
\newcommand{\SH}{\mathrm{SH}}
\newcommand{\so}{\mathfrak{so}}
\DeclareMathOperator{\td}{td}
\DeclareMathOperator{\Sym}{Sym}
\DeclareMathOperator{\Pic}{Pic}

\title[The Second Betti Number of Hyperk\"ahler Sixfolds]
{A Sharp Bound for the Second Betti Number\\of Hyperk\"ahler Sixfolds}
\author{Jie Fu}
\address{Jie Fu, Qiuzhen College, Tsinghua University, Beijing, China}
\email{fu-j21@mails.tsinghua.edu.cn}
\author{Bin Wang}
\address{Bin Wang, Hetao Institute of Mathematics and Interdisciplinary Sciences, Shenzhen, China}
\email{matph92@gmail.com}
\author{Shihao Wang}
\address{Shihao Wang, Qiuzhen College, Tsinghua University, Beijing, China}
\email{wangshihao25@mails.tsinghua.edu.cn}
\author{Zhiwei Zheng}
\address{Zhiwei Zheng, Yau Mathematical Sciences Center, Tsinghua University, Beijing, China}
\email{zhengzhiwei@mail.tsinghua.edu.cn}
\date{}

\begin{document}

\begin{abstract}
Let $X$ be a smooth compact irreducible hyperk\"ahler manifold of dimension six. We prove that the second Betti number satisfies $b_2(X) \le 23$. Furthermore, we show that the third cohomology vanishes whenever $b_2(X) \ge 19$, and that all odd cohomology vanishes when $b_2(X) = 23$. If the second Chern class lies in the subalgebra generated by $H^2(X,\Q)$, then $b_2(X) \in \{5,8,23\}$. The proof uses the Looijenga--Lunts--Verbitsky decomposition, curvature inequalities for Segre classes, and the Hodge--Riemann bilinear relations on the projective cotangent bundle.
\end{abstract}

\maketitle

\tableofcontents

\section{Introduction}\label{sec:introduction}

Let $X$ be a smooth compact irreducible hyperk\"ahler manifold of dimension six (an \emph{IHS sixfold} for short). By definition, $X$ is simply connected, K\"ahler, and the space of global holomorphic two-forms $H^0(X,\Omega_X^2)$ is spanned by a nondegenerate symplectic form. We denote its rational Betti numbers by $b_j(X) = \dim_\Q H^j(X,\Q)$. 

Establishing universal upper bounds on Betti numbers is a central open problem about hyperk\"ahler manifolds. In dimension six, the three known deformation types---the Hilbert scheme $\mathrm{K3}^{[3]}$, the generalized Kummer sixfold $\mathrm{Kum}_3$, and O'Grady's sixfold $\mathrm{OG6}$---exhibit second Betti numbers $b_2 = 23, 7$, and $8$, respectively. Consequently, any universal upper bound in dimension six must be at least $23$. The lower bound $b_2(X) \ge 3$ follows from $h^{2,0}(X) = h^{0,2}(X) = 1$ together with the existence of a K\"ahler class.

Cohomological relations in hyperk\"ahler geometry arise from multiple sources. First, in dimension six, Salamon's linear relation~\cite[Theorem~4.1]{Salamon} gives
\begin{equation}\label{eq:intro-salamon}
  b_6 = 6b_4 - 16b_3 + 30b_2 + 70.
\end{equation}
which follows from Hirzebruch--Riemann--Roch formula, the vanishing of $c_1(X)$, and the Hodge symmetries induced by the holomorphic symplectic form. Second, the total cohomology $H^*(X,\C)$ admits an action of the Looijenga--Lunts--Verbitsky (LLV) Lie algebra $\mathfrak{g} \simeq \so(b_2(X)+2,\C)$ (see~\cite[Proposition~4.5(i)]{LL} and~\cite[Theorem~1.4]{Verbitsky}), which imposes strong representation-theoretic constraints on the Betti vector.

In dimension four, Guan~\cite[Main Theorem]{Guan} proved that $b_2 \in \{3,4,5,6,7,8,23\}$, and that $b_2 = 23$ forces $(b_3, b_4) = (0, 276)$. More recently, Debarre, Huybrechts, Macr\`i, and Voisin~\cite[Theorem~1.5]{DHMV} proved that an irreducible hyperk\"ahler fourfold admitting classes $l,m \in H^2(X,\Z)$ with $\int_X l^4 = 0$ and $\int_X l^2m^2 = 2$ is of $\mathrm{K3}^{[2]}$ deformation type. Assuming the conjectural formula in~\cite[Conjecture~1.2]{BS} for the generalized Fujiki constant of $\operatorname{ch}_4$, Beckmann and Song~\cite[Corollary~1.3]{BS} obtained $b_2 \in \{5,6,7,23\}$.

In dimension six, Sawon~\cite[Section~3]{Sawon6} established $b_2 \le 23$ assuming that exterior-power summands of order at least four do not occur in the LLV decomposition. Kim and Laza~\cite[Main Theorem and Remarks~1.1--1.2]{KL} explained the limitation of earlier representation-theoretic arguments and obtained dimension-dependent bounds under the LLV highest-weight conjecture; for sixfolds, their bound $b_2 \le 23$ follows assuming $\sum_i |\mu_i| \le 3$ for the highest weight $\mu$ of every LLV summand. Green, Kim, Laza, and Robles~\cite[Theorem~1.2]{GKLR} determined the explicit LLV decompositions for all known deformation classes. They also established unconditional constraints on the first two highest-weight coordinates~\cite[Proposition~2.34]{GKLR} and formulated a stronger conjecture on the sum of all highest-weight coordinates~\cite[Conjecture~1.19]{GKLR}.

In this paper we prove the following unconditional bound.

\begin{theorem}\label{thm:main}
Let $X$ be a smooth compact irreducible hyperk\"ahler sixfold. Then
\begin{equation}
  3 \le b_2(X) \le 23.
\end{equation}
The upper bound is sharp and attained by manifolds of $\mathrm{K3}^{[3]}$ deformation type.
\end{theorem}

The same methods give restrictions on odd cohomology when $b_2(X)$ is large.

\begin{theorem}\label{thm:oddendpoint}
Let $X$ be a smooth compact irreducible hyperk\"ahler sixfold.
\begin{enumerate}
  \item[\textup{(i)}] If $b_2(X) \ge 19$, then $H^3(X,\Q) = 0$.
  \item[\textup{(ii)}] If $b_2(X) = 23$, then $H^{\mathrm{odd}}(X,\Q) = 0$.
\end{enumerate}
\end{theorem}

More precisely, we establish that $b_5(X) \in \{0, 1024, 2048\}$ when $b_2(X) = 19$, and $b_5(X) \in \{0, 1024\}$ for $20 \le b_2(X) \le 22$; see Corollary~\ref{cor:fifthbetti}.

Let $\SH^4(X,\Q)$ denote the image of $\Sym^2 H^2(X,\Q)$ inside $H^4(X,\Q)$ under the cup product. Under the purity hypothesis $c_2(X) \in \SH^4(X,\Q)$, the equality case of Beckmann--Song's inequality~\cite[Proposition~2.3]{BS} gives arithmetic restrictions on $b_2(X)$. At $b_2(X) = 23$, this hypothesis determines the full Betti vector.

\begin{theorem}\label{thm:pure}
Let $X$ be a smooth compact irreducible hyperk\"ahler sixfold. If $c_2(X) \in \SH^4(X,\Q)$, then
\begin{equation}
  b_2(X) \in \{5, 8, 23\}.
\end{equation}
If, in addition, $b_2(X) = 23$, then the full Betti vector of $X$ is uniquely determined as
\begin{equation}
  (b_0, \ldots, b_{12}) = (1, 0, 23, 0, 299, 0, 2554, 0, 299, 0, 23, 0, 1).
\end{equation}
\end{theorem}

We note that the condition $c_2(X) \in \SH^4(X,\Q)$ holds for $\mathrm{K3}^{[3]}$ and $\mathrm{OG6}$, but fails for $\mathrm{Kum}_3$; see~\cite[Theorem~1.1 and Examples~2.12--2.13]{BS}.

To prove Theorem~\ref{thm:main}, we assume for contradiction that $b_2(X) \ge 24$. We express $H^*(X,\C)$ as a direct sum of irreducible LLV representations. Salamon's relation and the Hirzebruch--Riemann--Roch theorem yield a system of linear equations in the non-negative integer multiplicities of these summands. Differential-geometric considerations provide two sign constraints on the Segre numbers: for any K\"ahler class $h$, the horizontal curvature properties of hyperk\"ahler metrics imply that $\int_X s_4(T_X) h^2 \ge 0$ and $\int_X s_6(T_X) \le 0$. Setting $R = 5760 \int_X \sqrt{\td_X}$ and $K = \int_X s_6(T_X)$, these inequalities give $R \ge 1440$ and $K \le 0$. Together with the multiplicity constraints, they eliminate all values $b_2(X) \ge 25$ and leave the remaining case
\[
  b_2(X) = 24, \qquad 2077 \le R \le 2170, \qquad -624 \le K \le 0.
\]

To exclude $b_2(X) = 24$, we pass to a very general deformation of $X$. On such a deformation, every proper irreducible compact analytic subvariety of $Y = \mathbb{P}(\Omega_X)$ is contained within a fiber of the natural projection $Y \to X$. Combining this property with the numerical bounds on $R$ and $K$, we use the numerical criterion of Demailly and P\u{a}un~\cite[Theorem~0.1]{DemaillyPaun} to construct a K\"ahler class $L = \zeta + h$, where $\zeta = c_1(\mathcal{O}_Y(1))$ and $h$ is K\"ahler with $q_X(h) = 3/4$. By the Hodge--Riemann bilinear relations, the Lefschetz intersection form $\int_Y \alpha \beta L^7$ must have negative index one on the four-dimensional subspace generated by $\zeta^2, \zeta h, h^2$, and the dual Beauville--Bogomolov--Fujiki class. An explicit computation using the bounds on $R$ and $K$ produces a negative-definite two-plane inside this subspace, yielding the desired contradiction.

For the vanishing of $H^3$, the signed Poincar\'e pairing forces the relevant odd multiplicities to be even at $b_2 = 19,20$. Together with the Segre inequalities and the LLV identities, this excludes nonzero third cohomology for $b_2 \ge 19$. The same identities give the bounds on $b_5$.

\medskip
\noindent\textbf{Recent developments on SYZ Conjecture.}
Engel, Filipazzi, Greer, Mauri, and Svaldi~\cite[Theorem~B and Corollary~1.3]{EFGMS} established the finiteness of locally trivial deformation types in fixed dimension for primitive symplectic varieties with Lagrangian fibrations. Their Corollary~1.3 extends this conclusion to $b_2 \ge 5$ under the hyperk\"ahler SYZ conjecture. Recent preprints by OpenAI~\cite[Theorem~1.1 and Corollary~1.2]{OAISYZ} and Engel--Mauri~\cite[Theorem~1.1 and Corollary~1.2]{EMSYZ} present proofs of SYZ statements and the corresponding finiteness consequences. The proofs in this paper do not use these SYZ results.

\medskip
\noindent\textbf{Structure of the paper.}
Section~\ref{sec:curvature} establishes the curvature inequalities for Segre classes and computes the normalized Fujiki constants.
Section~\ref{sec:llv} describes the possible LLV summands, expresses the Betti and Hodge numbers in terms of their multiplicities, and establishes parity constraints on odd cohomology.
Section~\ref{sec:reduction} proves the vanishing of $H^3$ for $b_2 \ge 19$ and establishes $b_2 \le 24$.
Section~\ref{sec:geometry} constructs the K\"ahler class on $\mathbb{P}(\Omega_X)$ and excludes the endpoint $b_2 = 24$ via the Hodge--Riemann relations.
Section~\ref{sec:endpoint} details the bounds on $b_5$ and completes the proofs of Theorems~\ref{thm:oddendpoint} and~\ref{thm:pure}.

\medskip
\noindent\textbf{Acknowledgements.}
This research was conducted in collaboration with large language model AI systems, which assisted in exploring proof strategies, refining proofs, verifying calculations, and identifying relevant literature. All contents are verified and finalized by the authors, who take full responsibility.

\section{Segre Classes and Fujiki Constants}\label{sec:curvature}

We establish sign constraints on the Segre numbers from hyperk\"ahler curvature and compute the generalized Fujiki constants in a fixed normalization. These formulas give the numerical inequalities used in the subsequent LLV analysis.

Throughout this section, let $X$ denote an IHS sixfold. We write $c_j = c_j(T_X)$ for its Chern classes and suppress explicit reference to $X$ when context allows. Recall that the formal Chern roots $x_1, \ldots, x_6$ determine the Todd class
\begin{equation}
  \td_X = \td(T_X) = \prod_{i=1}^6 \frac{x_i}{1 - e^{-x_i}},
\end{equation}
and we define $\sqrt{\td_X}$ as the unique formal square root whose degree-zero term is $1$. The Segre classes $s_j(T_X) \in H^{2j}(X,\Q)$ are generated by the formal inverse of the total Chern class,
\begin{equation}
  s(T_X) = 1 + \sum_{j=1}^6 s_j(T_X) = c(T_X)^{-1}.
\end{equation}
Because the tangent bundle $T_X$ admits a nondegenerate holomorphic symplectic form, its odd Chern classes vanish in rational cohomology. In terms of the even Chern classes, the Segre classes take the form
\begin{equation}
  s_2 = -c_2, \qquad s_4 = c_2^2 - c_4, \qquad s_6 = -c_2^3 + 2c_2 c_4 - c_6.
\end{equation}
We define two fundamental numerical invariants:
\begin{equation}
  R = 5760 \int_X \sqrt{\td_X}, \qquad K = \int_X s_6(T_X).
\end{equation}
By Sawon~\cite[Section~2]{SawonTop}, the degree-twelve component of $\sqrt{\td_X}$ is
\begin{equation}
  \bigl[\sqrt{\td_X}\bigr]_{12} = \frac{31c_2^3 - 44c_2c_4 + 16c_6}{967680},
\end{equation}
which implies
\begin{equation}\label{eq:Rchern}
  R = \frac{1}{168} \int_X \bigl(31c_2^3 - 44c_2c_4 + 16c_6\bigr).
\end{equation}
It is a theorem of Hitchin and Sawon~\cite[formula~(1) and Section~6]{HS} that $R > 0$ for every irreducible hyperk\"ahler manifold.

\begin{definition}
Let $\sigma \in H^0(X, \Omega_X^2)$ be a holomorphic symplectic form normalized such that $\int_X (\sigma \bar{\sigma})^3 = 1$. The associated quadratic form on $H^2(X,\R)$ is defined by
\begin{equation}
  q_\sigma(h) = \frac{3}{2}\int_X h^2 (\sigma \bar{\sigma})^2 - 2 \left(\int_X h \sigma^2 \bar{\sigma}^3\right)\left(\int_X h \sigma^3 \bar{\sigma}^2\right).
\end{equation}
The \emph{Beauville--Bogomolov--Fujiki (BBF) form} $q_{\mathrm{BBF}}$ is the unique positive scalar multiple of $q_\sigma$ whose symmetric bilinear polarization is integral and primitive on $H^2(X,\Z)$.
\end{definition}
By Beauville~\cite[Th\'eor\`eme~5(a) and its proof]{Beauville}, $q_{\mathrm{BBF}}$ has signature $(3, b_2(X)-3)$ and is positive on the K\"ahler cone.

\begin{definition}
Following Nieper~\cite[Definition~17]{Nieper}, the \emph{characteristic quadratic form} $\lambda_X$ on $H^2(X,\R)$ is defined by
\begin{equation}
  \lambda_X(h) =
  \begin{cases}
    \displaystyle \frac{12}{5} \frac{\int_X h^6}{\int_X c_2 h^4}, & \text{if } \int_X c_2 h^4 \ne 0,\\[3mm]
    0, & \text{if } \int_X c_2 h^4 = 0.
  \end{cases}
\end{equation}
\end{definition}
Nieper~\cite[Proposition~10]{Nieper} proved that $\lambda_X$ is a positive multiple of $q_{\mathrm{BBF}}$. Moreover, Nieper~\cite[formula~(5.18)]{Nieper} established that
\begin{equation}
  \frac{1}{720}\int_X h^6 = \left(\int_X \sqrt{\td_X}\right) \lambda_X(h)^3.
\end{equation}
Throughout the paper, we work with the convenient normalization $q_X := \lambda_X / 2$. This form is a positive multiple of the primitive integral BBF form, and the Fujiki constants below are taken with respect to $q_X$. Under this choice, the top self-intersection formula simplifies to
\begin{equation}\label{eq:fujiki-normalized}
  \int_X h^6 = R \, q_X(h)^3.
\end{equation}
Hence, the scalar $R$ coincides with the Fujiki constant associated with the normalized form $q_X$.

We denote by $\mathbb{P}(T_X)$ and $\mathbb{P}(\Omega_X)$ the projectivizations of the tangent and cotangent bundles of $X$, respectively. Since $X$ is a hyperk\"ahler manifold, they are isomorphic as projective bundles over $X$.

\begin{lemma}\label{lem:projectivebundle}
Let $X$ be an IHS sixfold, and let $\pi \colon Y = \mathbb{P}(\Omega_X) \to X$ denote the projective cotangent bundle. Under the natural identification of $\mathbb{P}(T_X)$ with $Y$ over $X$, let $\mathcal{O}_Y(1)$ be the dual of the tautological line bundle on $\mathbb{P}(T_X)$, and write $\zeta = c_1(\mathcal{O}_Y(1))$. Then
\begin{equation}
  \pi_*(\zeta^{5+j}) = s_j(T_X), \qquad \text{for all } 0 \le j \le 6.
\end{equation}
\end{lemma}
\begin{proof}
Consider the Euler sequence on $Y$. The quotient bundle $\pi^* T_X / \mathcal{O}_Y(-1)$ has rank five; hence, its sixth Chern class vanishes identically:
\begin{equation}
  \zeta^6 + \pi^* c_1(T_X)\zeta^5 + \dots + \pi^* c_6(T_X) = 0.
\end{equation}
Fibre integration yields $\pi_*(\zeta^5) = 1$ and $\pi_*(\zeta^\ell) = 0$ for all $0 \le \ell < 5$. Multiplying the vanishing relation by $\zeta^{j-1}$ and applying pushforward gives
\begin{equation}
  \sum_{i=0}^j c_i(T_X) \pi_*(\zeta^{5+j-i}) = 0, \qquad 1 \le j \le 6.
\end{equation}
Comparing this recurrence relation with the identity $c(T_X) s(T_X) = 1$ completes the proof.
\end{proof}

By the projection formula, we deduce
\begin{equation}
  \int_Y \zeta^9 (\pi^* h)^2 = \int_X s_4(T_X) h^2, \qquad \int_Y \zeta^{11} = \int_X s_6(T_X) = K.
\end{equation}

\begin{definition}\label{def:fujikiconstant}
Let $a \in H^{4j}(X,\Q)$ be a deformation-invariant characteristic class, with $0 \le j \le 3$. The generalized Fujiki relation~\cite[Theorem~2.1]{BS} guarantees the existence of a constant $C_X(a)$ relative to $q_X$, called the \emph{generalized Fujiki constant}, satisfying
\begin{equation}
  \int_X a h^{6-2j} = C_X(a) q_X(h)^{3-j}, \qquad \text{for all } h \in H^2(X,\R).
\end{equation}
\end{definition}
In particular, $C_X(1) = R$ and $C_X(s_6) = K$.

\begin{proposition}\label{prop:signedsegre}
Let $X$ be a hyperk\"ahler manifold of dimension $m=2n$ equipped with a hyperk\"ahler metric $g$ and associated K\"ahler form $\omega$. Let $\mathcal{S}_j(T_X,g)$ denote the Segre forms constructed from the Chern connection of $g$. Then, for all $0 \le k \le n$,
\begin{equation}
  (-1)^k \mathcal{S}_{2k}(T_X,g) \wedge \omega^{m-2k} \ge 0.
\end{equation}
Consequently, for every compact IHS sixfold,
\begin{equation}
  (-1)^k C_X(s_{2k}) \ge 0, \qquad \text{for all } 0 \le k \le 3.
\end{equation}
\end{proposition}
\begin{proof}
Fix a point $x \in X$, and choose a $g$-unitary frame $e_1, \ldots, e_m$ of $T_x^{1,0}X$ such that the parallel holomorphic symplectic form is represented by the standard symplectic matrix $J$ with $J^2 = -I$. Let $\theta^1, \ldots, \theta^m$ be the dual coframe, with $\omega_x = i\sum_i \theta^i \wedge \bar{\theta}^i$, and write the Chern curvature as $\Theta(T_X,g) = \sum_{i,j} M_{i\bar{j}}\theta^i \wedge \bar{\theta}^j$. For a unit vector $v \in T_x^{1,0}X$, consider the curvature endomorphism
\begin{equation}
  S_v = \sum_{i,j=1}^m v_i \bar{v}_j M_{i\bar{j}}, \qquad \text{where } M_{i\bar{j}}e_k = \sum_{\ell=1}^m R_{i\bar{j}k\bar{\ell}}e_\ell.
\end{equation}
The matrix $S_v$ is Hermitian and satisfies $S_v^t J + J S_v = 0$. Hence, $S_v$ is similar to $-S_v^t$, and its eigenvalues occur in opposite pairs.

By Diverio~\cite[Proposition~1.1 and Section~2]{Diverio}, the Segre forms are represented by fiber integrals of powers of the curvature form $\Xi = \omega_{\mathrm{FS}} - \eta_v$ of the dual tautological line bundle on $\mathbb{P}(T_X)$. The horizontal form at $[v]$ is
\begin{equation}
  \eta_v = \frac{i}{2\pi}\sum_{i,j=1}^m H_{i\bar{j}}(v)\theta^i \wedge \bar{\theta}^j,
  \qquad H_{i\bar{j}}(v) = \sum_{k,\ell=1}^m R_{i\bar{j}k\bar{\ell}}v_k\bar{v}_\ell.
\end{equation}
The K\"ahler curvature symmetry $R_{i\bar{j}k\bar{\ell}} = R_{k\bar{\ell}i\bar{j}}$ gives $H_{i\bar{j}}(v) = (S_v)_{ji}$. Thus, $H(v) = S_v^t$, and the eigenvalues of $\eta_v$ relative to $\omega$ occur in pairs $\pm\lambda_1(v), \ldots, \pm\lambda_n(v)$, with the factor $1/(2\pi)$ included in this definition.

The elementary symmetric polynomial of order $2k$ in these paired eigenvalues equals $(-1)^k e_k(\lambda_1(v)^2, \ldots, \lambda_n(v)^2)$. Taking the wedge product with $\omega^{m-2k}$ and evaluating over the fiber $\mathbb{P}(T_x X)$, as in Diverio~\cite[Proposition~3.3]{Diverio}, yields
\begin{equation}
  (-1)^k \mathcal{S}_{2k}(T_X,g) \wedge \omega^{m-2k} = \frac{\binom{m-1+2k}{2k}}{\binom{m}{2k}} \left( \int_{\mathbb{P}(T_x X)} e_k(\lambda_1(v)^2, \ldots, \lambda_n(v)^2) \, \omega_{\mathrm{FS}}^{m-1} \right) \omega^m.
\end{equation}
Since the integrand involves elementary symmetric polynomials of non-negative real numbers and the fiber measure is positive, with $\int_{\mathbb{P}(T_xX)}\omega_{\mathrm{FS}}^{m-1} = 1$, the differential form is non-negative. Every K\"ahler class on a compact IHS manifold has a Ricci-flat hyperk\"ahler representative. Integration over $X$ and the generalized Fujiki relation then give $(-1)^k C_X(s_{2k}) \ge 0$.
\end{proof}

\begin{lemma}\label{lem:fujiki}
In the normalization $q_X = \lambda_X/2$, we have
\begin{equation}
  C_X(1) = R, \quad C_X(c_2) = \frac{6R}{5}, \quad C_X(c_2^2) = \frac{4R + 23040}{5}, \quad C_X(c_4) = \frac{40320 - 8R}{5}.
\end{equation}
In particular, $C_X(s_4) = \frac{12}{5}(R - 1440)$.
\end{lemma}
\begin{proof}
Let $\mathrm{RR}_X(t)$ denote the polynomial satisfying $\int_X e^h \td_X = \mathrm{RR}_X(q_X(h))$. For a line bundle $\mathcal{L}$, Hirzebruch--Riemann--Roch gives
\[
  \chi(X,\mathcal{L}) = \mathrm{RR}_X(q_X(c_1(\mathcal{L}))).
\]
In the notation of Debarre~\cite[Proposition~2.1]{Debarre}, we have $\mathrm{RR}_X(t) = Q_{RR,X}(2t)$, since $q_X = \lambda_X/2$. The symmetry in that proposition therefore gives $\mathrm{RR}_X(-t-2) = -\mathrm{RR}_X(t)$ in dimension six. Thus, $\mathrm{RR}_X(t) = A(t+1)^3 + B(t+1)$. Comparing the leading coefficient $A = R/720$ with $\chi(\mathcal{O}_X) = 4$ yields
\begin{equation}\label{eq:RRsix}
  \mathrm{RR}_X(t) = \frac{R}{720}t^3 + \frac{R}{240}t^2 + \left(\frac{R}{360} + 4\right)t + 4.
\end{equation}
Equating coefficients with the Chern character expansion establishes $C_X(c_2) = 6R/5$ and $3C_X(c_2^2) - C_X(c_4) = 4R + 5760$. By Beckmann--Song~\cite[Corollary~2.6]{BS}, we also have $7C_X(c_2^2) - 4C_X(c_4) = 12R$. Solving this linear system gives the explicit formulas for $C_X(c_2^2)$ and $C_X(c_4)$, from which $C_X(s_4) = C_X(c_2^2 - c_4)$ follows.
\end{proof}

Applying Proposition~\ref{prop:signedsegre} to Lemma~\ref{lem:fujiki} gives the fundamental lower bound
\begin{equation}\label{eq:segre4lower}
  R \ge 1440,
\end{equation}
together with the top Segre constraint
\begin{equation}\label{eq:segre6budget}
  K \le 0.
\end{equation}
Combining Beckmann--Song's inequality~\cite[Propositions~2.3--2.4]{BS} with Lemma~\ref{lem:fujiki}, we obtain
\begin{equation}\label{eq:sixfoldC2UpperCore}
  R \le \frac{2880(b_2 + 4)}{b_2 + 1},
\end{equation}
with equality holding if and only if $c_2(X) \in \SH^4(X,\Q)$.

\begin{lemma}\label{lem:chernchi}
Let $X$ be an IHS sixfold and set $\chi^p = \chi(X, \Omega_X^p)$. Then
\begin{equation}\label{eq:sixfoldChernChiCore}
  \begin{aligned}
    \int_X c_2^3   &= 29088 - 184\chi^1 - 8\chi^2, \\
    \int_X c_2 c_4 &= 5472 - 208\chi^1 - 8\chi^2, \\
    \int_X c_6     &= 144 - 16\chi^1 + 4\chi^2.
  \end{aligned}
\end{equation}
Consequently,
\begin{equation}\label{eq:Schi}
  R = 3948 + 19\chi^1 + \chi^2, \qquad K = 12(2424 + \chi^1 - R).
\end{equation}
\end{lemma}
\begin{proof}
Applying Hirzebruch--Riemann--Roch to the sheaves $\mathcal{O}_X, \Omega_X^1$, and $\Omega_X^2$, and utilizing the vanishing of odd Chern classes together with $\chi(\mathcal{O}_X) = 4$, we obtain the linear equations
\begin{equation}
  \begin{aligned}
    241920        &= 10a - 9d + 2e, \\
    60480 \chi^1  &= 60a - 306d - 492e, \\
    60480 \chi^2  &= 150a - 1143d + 13134e,
  \end{aligned}
\end{equation}
where $(a,d,e) = \int_X (c_2^3, c_2 c_4, c_6)$. Solving this system yields~\eqref{eq:sixfoldChernChiCore}. Substituting these values into~\eqref{eq:Rchern} and noting that $K = -a + 2d - e$ produces~\eqref{eq:Schi}.
\end{proof}

\section{Representation Theory of the LLV Algebra}\label{sec:llv}

We describe the possible LLV summands and translate their graded characters into Betti and Hodge number formulas. We also establish parity constraints on the odd multiplicities and express $R$ and $K$ through the LLV decomposition.

Let $X$ be an IHS sixfold, let $\mathfrak{g} \simeq \so(b_2(X)+2, \C)$ denote the LLV Lie algebra acting on $H^*(X,\C)$, and let $V = \C^{b_2(X)+2}$ be its standard representation. Setting $r = \lfloor(b_2+2)/2\rfloor$, the Lie algebra $\mathfrak{g}$ has rank $r$ and is of Cartan type $B_r$ for odd $b_2$ and type $D_r$ for even $b_2$.

For any dominant weight $\mu = (\mu_0, \ldots, \mu_{r-1})$ in standard orthonormal coordinates, let $V_\mu$ denote the corresponding irreducible $\mathfrak{g}$-module. Following~\cite[Section~2.2.1]{GKLR}, any weight $\nu$ occurring in the Hodge component $H^{p,q}(X)$ satisfies
\begin{equation}
  \nu_0 = \frac{p+q-6}{2}, \qquad \nu_1 = \frac{p-q}{2}.
\end{equation}
The subring $\SH \subset H^*(X,\C)$ generated by $H^2(X,\C)$ forms the irreducible Verbitsky component $V_{(3,0,\ldots,0)}$~\cite[Theorem~2.15]{GKLR}.

\begin{lemma}\label{lem:evenLLVsummands}
Let $X$ be an IHS sixfold with $b_2(X) \ge 19$. Every non-Verbitsky irreducible LLV summand $V_\mu$ satisfies the unconditional bound
\[
  \mu_0 + \mu_1 \le 2.
\]
Its highest-weight coordinates are integral in even cohomology and belong to $\Z + \frac12$ in odd cohomology. In particular, every non-Verbitsky irreducible summand in even cohomology is isomorphic to the trivial module $\mathbf{1}$ or to one of the following exterior-power representations:
\[
  \begin{cases}
    E_k = \Lambda^k V,\quad 1 \le k \le r, & \text{in type } B_r,\\
    E_k = \Lambda^k V,\quad 1 \le k < r,\quad E_r^+,\ E_r^-, & \text{in type } D_r,
  \end{cases}
\]
where $\Lambda^r V = E_r^+ \oplus E_r^-$ in type $D_r$.
\end{lemma}
\begin{proof}
The highest-weight bound and the integrality assertions follow from Green--Kim--Laza--Robles~\cite[Propositions~2.34--2.35]{GKLR}. For integral highest weights, dominance together with $\mu_0 + \mu_1 \le 2$ leaves only the trivial module, the exterior-power representations listed above, and
\[
  V_{(2,0,\ldots,0)} = \Sym^2_0 V,
\]
the trace-free symmetric square of $V$. The lowest graded piece of $\Sym^2_0 V$ is one-dimensional and has $\nu_0 = -2$. Since cohomological degree is $6 + 2\nu_0$, this piece would contribute to $H^2(X,\C)$. But $H^2(X,\C)$ is already exhausted by $\SH$, so $\Sym^2_0 V$ cannot occur as an additional summand. This proves the claimed list.
\end{proof}

We use \emph{spin representation} for an irreducible module with highest weight $(\frac{1}{2}, \ldots, \frac{1}{2})$, and \emph{vector-spin representation} for one with highest weight $(\frac{3}{2}, \frac{1}{2}, \ldots, \frac{1}{2})$. The vector-spin module is obtained by tensoring a spin module with the standard module $V$ and removing one spin summand. In type $B_r$, there is a single spin module $S$, and the spin summand in $V \otimes S$ is again $S$. In type $D_r$, there are two half-spin modules $S^+$ and $S^-$, distinguished by the sign of the last highest-weight coordinate. We call both of them spin modules. Here $V \otimes S^+$ contains $S^-$, while $V \otimes S^-$ contains $S^+$; the remaining summand in each case is a vector-spin module. A module's contribution to $H^j$ means the dimension of its degree-$j$ part. In type $D_r$, we add the multiplicities of the two half-spin modules, and likewise add those of the two vector-spin modules.

\begin{lemma}\label{prop:oddgrowth}
Let $X$ be an IHS sixfold with $b_2 \ge 19$. Every odd LLV summand is either a spin representation or a vector-spin representation. Set $d_{b_2} = 2^{\lfloor(b_2-1)/2\rfloor}$. Each spin representation contributes $(0, d_{b_2})$ to $(H^3, H^5)$, while each vector-spin representation contributes $(d_{b_2}, b_2 d_{b_2})$. If $u$ and $v$ denote the total multiplicities of vector-spin and spin summands, respectively, then
\begin{equation}\label{eq:sixfoldOddSemigroupCore}
  b_3 = d_{b_2} u, \qquad b_5 = d_{b_2}(b_2 u + v), \qquad u, v \in \Z_{\ge 0}; \qquad w := b_5 - b_2 b_3 \ge 0.
\end{equation}
Moreover, if $b_2 \in \{19, 20\}$, then $u$ and $v$ are both even integers.
\end{lemma}
\begin{proof}
By Lemma~\ref{lem:evenLLVsummands}, the odd highest weights have half-integral coordinates. Dominance and $\mu_0 + \mu_1 \le 2$ force $\mu_0 \in \{\frac12,\frac32\}$ and $|\mu_i| = \frac12$ for $i \ge 1$. These are precisely the spin and vector-spin highest weights described above.

For a graded LLV module $M$, write $M_a$ for the part with first weight coordinate $\nu_0 = a$, which contributes to cohomological degree $6 + 2a$. The standard module $V$ has grades $-1,0,1$, with dimensions $1,b_2,1$, respectively. Let $S_*$ denote $S$ in type $B_r$ or either one of $S^\pm$ in type $D_r$. Its weights are $(\pm\frac12,\ldots,\pm\frac12)$, each with multiplicity one; all sign choices occur in type $B_r$, whereas in type $D_r$ the parity of the number of minus signs is fixed by the chirality. Fixing the first coordinate therefore gives
\[
  \dim (S_*)_{-1/2} = \dim (S_*)_{1/2}
  = \begin{cases}
      2^{r-1}, & \text{in type } B_r,\\
      2^{r-2}, & \text{in type } D_r,
    \end{cases}
  = d_{b_2}.
\]
There are no other grades, so each spin representation contributes $(0,d_{b_2})$ to $(H^3,H^5)$.

Write $T$ and $T^\pm$ for the vector-spin representations in types $B_r$ and $D_r$, respectively. The tensor-product decompositions are
\[
  V \otimes S \simeq T \oplus S \quad\text{in type } B_r,
  \qquad
  V \otimes S^\pm \simeq T^\pm \oplus S^\mp \quad\text{in type } D_r.
\]
First weight coordinates add in tensor products, so the two lowest grades are
\[
  \begin{aligned}
    (V \otimes S_*)_{-3/2}
      &= V_{-1} \otimes (S_*)_{-1/2},\\
    (V \otimes S_*)_{-1/2}
      &= \bigl(V_{-1} \otimes (S_*)_{1/2}\bigr)
         \oplus \bigl(V_0 \otimes (S_*)_{-1/2}\bigr),
  \end{aligned}
\]
of dimensions $d_{b_2}$ and $(b_2+1)d_{b_2}$. Removing the spin summand changes these dimensions to $d_{b_2}$ and $(b_2+1)d_{b_2}-d_{b_2} = b_2d_{b_2}$, respectively. Since the grades $-\frac32$ and $-\frac12$ correspond to $H^3$ and $H^5$, each vector-spin representation has the claimed contributions. Summing over the multiplicities $u$ and $v$ proves~\eqref{eq:sixfoldOddSemigroupCore}.

For the parity assertion, consider the signed Poincar\'e pairing on $H^{\mathrm{odd}}(X,\C)$:
\begin{equation}
  \Phi(\alpha,\beta) = (-1)^{(j-1)/2} \int_X \alpha \cup \beta, \qquad \alpha \in H^j(X,\C), \quad j \text{ odd}.
\end{equation}
This is a non-degenerate symmetric form by Poincar\'e duality and the sign convention. It is $\mathfrak{g}$-invariant because the Lefschetz $\mathfrak{sl}_2$ triples are skew-adjoint for $\Phi$~\cite[Section~1]{LL}.

Consequently, any irreducible summand admitting an invariant alternating form induces an alternating form on its multiplicity space, forcing its multiplicity to be even. By the invariant spinor pairings described in~\cite[Exercise~20.38(b)--(c)]{FultonHarris} and~\cite[Table~1.5.1]{DeligneSpinors}, the spin representation $S$ of $\so(21,\C)$ admits such a form. Since the standard representation $V$ has an invariant symmetric form, $V \otimes S$ has an invariant alternating form. Its two nonisomorphic self-dual summands, $S$ and the vector-spin representation, are orthogonal, so the restrictions of the form are non-degenerate. Thus, $u$ and $v$ are even when $b_2 = 19$. For $\so(22,\C)$, the two half-spin representations are dual to one another, as are the two vector-spin representations. The non-degenerate invariant pairing forces equal multiplicities within each dual pair; hence their total multiplicities $u$ and $v$ are even when $b_2 = 20$.
\end{proof}

\begin{remark}\label{rem:VoisinOddBounds}
Voisin~\cite[Theorems~2.3--2.4]{Voisin} proves exponential lower bounds for nonzero odd cohomology in arbitrary dimension. For an IHS sixfold, these give
\[
  b_3 > 0 \ \Longrightarrow\ b_3 \ge d_{b_2},
  \qquad
  b_3 = 0,\ b_5 > 0 \ \Longrightarrow\ b_5 \ge d_{b_2},
  \qquad d_{b_2} = 2^{\lfloor(b_2-1)/2\rfloor}.
\]
When $4 \mid b_2$, both lower bounds improve to $2^{b_2/2}$. Her argument uses cup-product maps and the universality of the Kuga--Satake construction.

In the six-dimensional range $b_2 \ge 19$, Lemma~\ref{prop:oddgrowth} recovers the basic exponential bounds directly from the LLV decomposition: $b_3 > 0$ forces $u \ge 1$, while $b_3 = 0$ and $b_5 > 0$ force $u = 0$ and $v \ge 1$. The exact formulas~\eqref{eq:sixfoldOddSemigroupCore} imply the divisibility $d_{b_2} \mid b_3,b_5$ and the joint constraint $b_5-b_2b_3 = d_{b_2}v \ge 0$. At $b_2 = 19,20$, the evenness of $u$ and $v$ raises both lower bounds to $2d_{b_2} = 1024$; for $b_2 = 20$, this agrees with Voisin's refinement. These multiplicity constraints will be combined below with the Segre inequalities.
\end{remark}

Let $\mathcal{I}$ index the irreducible non-Verbitsky exterior summands $E_i$, let $m_i \in \Z_{\ge 0}$ denote their multiplicities, and let $\tau \in \Z_{\ge 0}$ denote the multiplicity of the trivial representation. We write $b_j(E)$ and $h^{p,q}(E)$ for the contributions of a graded LLV module $E$ to the corresponding Betti and Hodge numbers. Set
\begin{equation}
  P_0(b_2) = b_2^2 - 60b_2 + 4068, \qquad f(b_2) = \frac{-b_2^3 + 15b_2^2 + 196b_2 + 420}{6}.
\end{equation}

\begin{lemma}\label{lem:sixfoldBudgetsCore}
Let $X$ be an IHS sixfold with $b_2 \ge 19$. Then
\begin{equation}\label{eq:Smodules}
  R = P_0(b_2) + \sum_{i \in \mathcal{I}} m_i \bigl(2b_4(E_i) - 23h^{3,1}(E_i)\bigr) + \frac{37-b_2}{2} b_3 - \frac{w}{2}.
\end{equation}
\begin{equation}\label{eq:Salamonmodules}
  \sum_{i \in \mathcal{I}} m_i \bigl(b_6(E_i) - 6b_4(E_i)\bigr) + \tau = f(b_2) - 16b_3.
\end{equation}
Furthermore, $b_4 = \binom{b_2+1}{2} + \sum_{i \in \mathcal{I}} m_i b_4(E_i)$ and
\begin{equation}\label{eq:chi1modules}
  \chi^1 = b_3 - 3b_2 + 6 - \sum_{i \in \mathcal{I}} m_i h^{3,1}(E_i).
\end{equation}
Finally, the Segre inequality $K \le 0$ is equivalent to
\begin{equation}\label{eq:oddSegrebudget}
  \frac{w}{2} \le b_2^2 - 57b_2 + 1638 + \frac{35-b_2}{2} b_3 + \sum_{i \in \mathcal{I}} m_i \bigl(2b_4(E_i) - 22h^{3,1}(E_i)\bigr).
\end{equation}
\end{lemma}
\begin{proof}
Hodge symmetry and the Lefschetz decomposition on $X$ imply
\begin{equation}
  \chi^1 = b_3 - 2b_2 + 4 - h^{3,1}, \qquad \chi^2 = 2 + 2h^{2,2} - \frac{b_3 + b_5}{2}, \qquad h^{2,2} = b_4 - 2 - 2h^{3,1}.
\end{equation}
The exterior power $E_k$ contributes
\begin{equation}
  b_4(E_k) = \binom{b_2}{k-1}, \qquad h^{3,1}(E_k) = \binom{b_2-2}{k-2}, \qquad b_6(E_k) = \binom{b_2}{k} + \binom{b_2}{k-2}.
\end{equation}
For either middle exterior summand $E_r^\pm$ in type $D_r$, these contributions are one half of those of $E_r$. The outer automorphism exchanging the two summands fixes the first two Cartan coordinates, so their graded Hodge characters agree. Substituting these contributions together with $b_4(\SH) = \binom{b_2+1}{2}$, $b_6(\SH) = \binom{b_2+2}{3}$, and $h^{3,1}(\SH) = b_2 - 2$ into Lemma~\ref{lem:chernchi} gives~\eqref{eq:Smodules} and~\eqref{eq:chi1modules}. Salamon's relation~\eqref{eq:intro-salamon} yields~\eqref{eq:Salamonmodules}. Lastly, substituting~\eqref{eq:Smodules} and~\eqref{eq:chi1modules} into the inequality $R \ge 2424 + \chi^1$ (which is equivalent to $K \le 0$) yields~\eqref{eq:oddSegrebudget}.
\end{proof}

\section{Numerical Reduction to \texorpdfstring{$b_2 \le 24$}{b2 <= 24}}\label{sec:reduction}

We combine the Segre inequalities with the LLV multiplicity identities to prove $b_2(X) \le 24$ and the vanishing of $H^3(X,\Q)$ for $b_2(X) \ge 19$. We also bound the numbers $R$ and $K$ in the remaining case $b_2(X) = 24$, preparing for the geometric argument in the next section.

For the coefficient comparisons below, we write
\begin{equation}
  \operatorname{Sal}(E) = b_6(E) - 6b_4(E), \qquad \mathcal{R}(E) = 2b_4(E) - 23h^{3,1}(E)
\end{equation}
for the contributions of an exterior summand $E$ to Salamon's relation~\eqref{eq:Salamonmodules} and to $R$ in~\eqref{eq:Smodules}, respectively. Setting $x = k(b_2+2-k)$, we have
\begin{equation}\label{eq:exteriorratios}
  \frac{\operatorname{Sal}(E_k)}{b_4(E_k)} = \frac{(b_2+1)(b_2+2)}{x} - 8, \qquad \frac{\mathcal{R}(E_k)}{b_4(E_k)} = 2 + \frac{23(b_2+1-x)}{b_2(b_2-1)}.
\end{equation}
For any $\gamma > 0$, the AM--GM inequality yields
\begin{equation}\label{eq:AMGMcomparison}
  \begin{aligned}
    \frac{\gamma\operatorname{Sal}(E_k) - \mathcal{R}(E_k)}{b_4(E_k)} &= \frac{\gamma(b_2+1)(b_2+2)}{x} + \frac{23x}{b_2(b_2-1)} - 8\gamma - 2 - \frac{23(b_2+1)}{b_2(b_2-1)} \\
      &\ge 2\sqrt{\frac{23\gamma(b_2+1)(b_2+2)}{b_2(b_2-1)}} - 8\gamma - 2 - \frac{23(b_2+1)}{b_2(b_2-1)}.
  \end{aligned}
\end{equation}

\begin{proposition}\label{prop:26}
Every smooth compact irreducible hyperk\"ahler sixfold satisfies $b_2(X) \le 26$.
\end{proposition}
\begin{proof}
Assume for contradiction that $b_2 \ge 27$. Setting $\gamma = 1/2$ in~\eqref{eq:AMGMcomparison}, we find that
\begin{equation}
  \mathcal{R}(E_i) < \frac{1}{2}\operatorname{Sal}(E_i) - \frac{1}{3}b_4(E_i)
\end{equation}
holds for all exterior summands. Indeed, for $z = b_2-27 \ge 0$, clearing the positive denominators in the required squared inequality gives
\[
  \begin{aligned}
    &414b_2(b_2-1)(b_2+1)(b_2+2) - \bigl(19b_2(b_2-1)+69(b_2+1)\bigr)^2 \\
    &\qquad = 53z^4+4652z^3+139454z^2+1521660z+2817036 > 0.
  \end{aligned}
\]
Thus the AM--GM lower bound in~\eqref{eq:AMGMcomparison} is strictly greater than $1/3$. Combining this inequality with~\eqref{eq:Smodules} and~\eqref{eq:Salamonmodules} gives
\begin{equation}
  R \le Q_0(b_2) - \frac{1}{3}\left(b_4 - \binom{b_2+1}{2}\right) + \frac{21-b_2}{2}b_3 - \frac{w}{2} - \frac{\tau}{2},
\end{equation}
where $Q_0(b_2) = \frac{-b_2^3 + 27b_2^2 - 524b_2 + 49236}{12}$. For $b_2 \ge 27$, the function $Q_0$ is strictly decreasing, with $Q_0(27) = 2924$. Since $f(b_2) < 0$, Salamon's relation~\eqref{eq:Salamonmodules} forces the existence of at least one summand with $\operatorname{Sal}(E_i) < 0$. However, $\operatorname{Sal}(E_k) > 0$ for all $1 \le k \le 4$; thus, we must have $k \ge 5$, which forces $b_4 - \binom{b_2+1}{2} \ge \binom{b_2}{4}$. For either middle exterior summand in type $D_r$, the same lower bound holds because $b_4(E_r^\pm) = \frac{1}{2}\binom{b_2}{r-1} \ge \binom{b_2}{4}$ for even $b_2 \ge 28$. Consequently,
\begin{equation}
  0 < R \le 2924 - \frac{1}{3} \binom{27}{4} = -2926,
\end{equation}
which is an contradiction.
\end{proof}

\begin{lemma}\label{prop:highbodd}
Let $X$ be an IHS sixfold with $b_2(X) \ge 19$. Then $b_3(X) = 0$.
\end{lemma}
\begin{proof}
Suppose first that $b_2 \ge 20$. Evaluating~\eqref{eq:AMGMcomparison} at $\gamma = 2/3$ shows that $\mathcal{R}(E_i) \le \frac{2}{3}\operatorname{Sal}(E_i)$ for all exterior summands. Indeed, for $z = b_2-20 \ge 0$,
\[
  \begin{aligned}
    &552b_2(b_2-1)(b_2+1)(b_2+2) - \bigl(22b_2(b_2-1)+69(b_2+1)\bigr)^2 \\
    &\qquad = 68z^4+4476z^3+99563z^2+779730z+692639 > 0,
  \end{aligned}
\]
so the AM--GM lower bound is positive. Hence,
\begin{equation}
  R \le U(b_2) - \frac{3b_2-47}{6} b_3, \qquad \text{where } U(b_2) = \frac{-b_2^3 + 24b_2^2 - 344b_2 + 37032}{9}.
\end{equation}
The polynomial $U(b_2)$ decreases monotonically for $b_2 \ge 20$, with $U(20) = 3528$. If $b_3 > 0$, Lemma~\ref{prop:oddgrowth} gives that $b_3 \ge 1024$. This implies $R \le 3528 - \frac{13}{6}(1024) = \frac{3928}{3} < 1440$, contradicting the lower bound $R \ge 1440$ from~\eqref{eq:segre4lower}.

Now suppose $b_2 = 19$. Setting $x = k(21-k)$, we verify that
\begin{equation}
  \frac{(34/57)\operatorname{Sal}(E_k) - \mathcal{R}(E_k) - h^{3,1}(E_k)}{b_4(E_k)} = \frac{(x-68)(11x-630)}{171x} \ge 0,
\end{equation}
because no allowable integer value of $x$ falls in the interval $(630/11, 68)$. Together with~\eqref{eq:Salamonmodules} and~\eqref{eq:oddSegrebudget}, this implies
\begin{equation}\label{eq:odd19budget}
  22528 u + 7296 v + 17\tau \le 33756.
\end{equation}
This forces $u \le 1$. By Lemma~\ref{prop:oddgrowth}, $u$ must be even, so $u = 0$, which proves $b_3 = 0$.
\end{proof}

\begin{lemma}\label{prop:endpointfamilies}
For $b_2 \in \{24, 25, 26\}$, only the exterior modules $E_1, E_2, E_3, E_4$ and the trivial representation can appear outside $\SH$. In each case, $m_4 = 1$, and the multiplicities satisfy:
\begin{align}
  b_2 = 24 \colon & \quad 18m_1 + 133m_2 + 392m_3 + \tau = 1232, \notag \\
                  & \quad R = 1939 + 2m_1 + 25m_2 + 46m_3 - b_5/2; \label{eq:endpoint24} \\
  b_2 = 25 \colon & \quad 19m_1 + 151m_2 + 525m_3 + \tau = 695, \notag \\
                  & \quad R = 1974 + 2m_1 + 27m_2 + 71m_3 - b_5/2; \label{eq:endpoint25} \\
  b_2 = 26 \colon & \quad 20m_1 + 170m_2 + 676m_3 + \tau = 5, \notag \\
                  & \quad R = 2036 + 2m_1 + 29m_2 + 98m_3 - b_5/2. \label{eq:endpoint26}
\end{align}
\end{lemma}
\begin{proof}
By Lemma~\ref{prop:highbodd}, $b_3 = 0$. For each $b_2 \in \{24, 25, 26\}$, we subtract $\lambda$ times~\eqref{eq:Salamonmodules} from~\eqref{eq:Smodules}, choosing $\lambda = \mathcal{R}(E_2)/\operatorname{Sal}(E_2)$. Set $\delta_k = \mathcal{R}(E_k) - \lambda\operatorname{Sal}(E_k)$ and $F = P_0(b_2) + \lambda f(b_2)$. Then
\begin{equation}
  R = F + \sum_{i \in \mathcal{I}} m_i\delta_i - \lambda\tau - b_5/2,
\end{equation}
where $\delta_i = \delta_r/2$ for either middle exterior summand in type $D_r$. For $k \ge 5$, the expression
\begin{equation}
  \frac{\delta_k}{b_4(E_k)} = 2 + 8\lambda - \frac{\lambda(b_2+1)(b_2+2)}{x} + \frac{23(b_2+1-x)}{b_2(b_2-1)}
\end{equation}
is negative at $x = 5(b_2-3)$ and decreasing for larger $x$, since
\[
  \lambda(b_2+1)(b_2+2)b_2(b_2-1) < 23\bigl[5(b_2-3)\bigr]^2
\]
holds at $b_2 = 24,25,26$. For an unsplit exterior power with $k \ge 5$, the inequality $b_4(E_k) \ge b_4(E_5)$ therefore gives $\delta_k \le \delta_5 < 0$. Either middle summand contributes $\delta_r/2 \le \delta_5/2$. Furthermore, direct substitution gives $\delta_k \le 0$ for $1 \le k \le 4$. The relevant numerical bounds are
\begin{equation}
  \begin{array}{c|cc}
    b_2 & F + \delta_5/2 & F + 2\delta_4 \\ \hline
    24 & -39420/19 & 151492/133 \\
    25 & -821777/302 & 155720/151 \\
    26 & -57428/17 & 16053/17
  \end{array}
\end{equation}
Hence, any non-zero multiplicity with $k \ge 5$, including either middle exterior summand, forces $R \le F + \delta_5/2 < 0$, a contradiction. Furthermore, since $\operatorname{Sal}(E_k) > 0$ for $1 \le k \le 3$ while $f(b_2) < 0$, we must have $m_4 \ge 1$. If $m_4 \ge 2$, then $R \le F + 2\delta_4 < 1440$, violating~\eqref{eq:segre4lower}. Thus, $m_4 = 1$, and substituting this value into the respective systems yields~\eqref{eq:endpoint24}--\eqref{eq:endpoint26}.
\end{proof}

\begin{lemma}\label{prop:24segre}
We have $b_2(X) \le 24$. If $b_2(X) = 24$, then
\begin{equation}\label{eq:24segrebudget}
  \begin{gathered}
    b_3 = b_5 = 0, \quad m_4 = 1, \quad 18m_1 + 133m_2 + 392m_3 + \tau = 1232, \\
    J := 2m_1 + 26m_2 + 68m_3 \ge 188, \\
    R = 1939 + 2m_1 + 25m_2 + 46m_3, \qquad K = 12(188 - J).
  \end{gathered}
\end{equation}
In particular, $R$ and $K$ satisfy
\begin{equation}\label{eq:projectiveHRrectangle}
  2077 \le R \le 2170, \qquad -624 \le K \le 0.
\end{equation}
\end{lemma}
\begin{proof}
By Proposition~\ref{prop:26}, we have $b_2 \le 26$, so we first exclude $b_2 = 25,26$. Lemma~\ref{prop:highbodd} gives $b_3 = 0$ in all the cases considered below.

At $b_2 = 26$, the nonnegativity of the multiplicities in~\eqref{eq:endpoint26} forces $m_1 = m_2 = m_3 = 0$ and $\tau = 5$. Substituting these values into~\eqref{eq:endpoint26}, \eqref{eq:chi1modules}, and~\eqref{eq:Schi} gives
\begin{equation}\label{eq:b26Kcontradiction}
  \frac{K}{12} = 40 + \frac{b_5}{2} > 0,
\end{equation}
contradicting~\eqref{eq:segre6budget}.

At $b_2 = 25$, the same identities, now using~\eqref{eq:endpoint25}, give
\begin{equation}\label{eq:b25Kidentity}
  \frac{K}{12} = 128 - 2m_1 - 28m_2 - 94m_3 + \frac{b_5}{2}.
\end{equation}
On the other hand, \eqref{eq:endpoint25} and $\tau \ge 0$ imply
\begin{equation}\label{eq:b25multiplicitybound}
  19m_1 + 151m_2 + 525m_3 \le 695, \qquad m_3 \le 1.
\end{equation}
Since $m_3$ is an integer, there are two cases. In each case, \eqref{eq:b25multiplicitybound} bounds $m_2$; eliminating $m_1$ with the same inequality then gives
\begin{equation}\label{eq:b25casebounds}
  \begin{aligned}
    m_3 = 0 &: \quad m_2 \le 4, \qquad 2m_1 + 28m_2 < 124, \\
    m_3 = 1 &: \quad m_2 \le 1, \qquad 2m_1 + 28m_2 + 94 \le 124.
  \end{aligned}
\end{equation}
Thus~\eqref{eq:b25casebounds}, combined with~\eqref{eq:b25Kidentity} and $b_5 \ge 0$, yields
\begin{equation}\label{eq:b25Kcontradiction}
  2m_1 + 28m_2 + 94m_3 \le 124, \qquad K \ge 48 > 0.
\end{equation}
This again contradicts~\eqref{eq:segre6budget}, so $b_2 \le 24$.

Now assume $b_2 = 24$. From~\eqref{eq:endpoint24}, we have
\begin{equation}\label{eq:b24multiplicitybound}
  18m_1 + 133m_2 + 392m_3 = 1232 - \tau \le 1232.
\end{equation}
The coefficients of $m_1,m_2,m_3$ in the expression for $R$ in~\eqref{eq:endpoint24} are at most $25/133$ times the corresponding coefficients in~\eqref{eq:b24multiplicitybound}. Since the multiplicities are nonnegative, these two equations imply
\begin{equation}\label{eq:b24Rwithb5}
  R \le 1939 + \frac{25}{133}\,1232 - \frac{b_5}{2} < 2171 - \frac{b_5}{2}.
\end{equation}
If $b_5 > 0$, Lemma~\ref{prop:oddgrowth} and~\eqref{eq:b24Rwithb5} give
\begin{equation}\label{eq:b24b5contradiction}
  b_5 \ge 2048, \qquad R < 1440,
\end{equation}
contradicting~\eqref{eq:segre4lower}. Hence $b_5 = 0$. Substituting this into~\eqref{eq:endpoint24}, \eqref{eq:chi1modules}, and~\eqref{eq:Schi}, we obtain
\begin{equation}\label{eq:b24RKidentities}
  \begin{gathered}
    R = 1939 + 2m_1 + 25m_2 + 46m_3, \\
    J := 2m_1 + 26m_2 + 68m_3, \qquad K = 12(188-J).
  \end{gathered}
\end{equation}
Combining~\eqref{eq:b24RKidentities} with~\eqref{eq:segre6budget} gives $J \ge 188$. Together with Lemmas~\ref{prop:highbodd} and~\ref{prop:endpointfamilies}, this proves~\eqref{eq:24segrebudget}.

Comparing coefficients with~\eqref{eq:b24multiplicitybound} again, we have
\begin{equation}\label{eq:b24RJupper}
  \begin{aligned}
    R - 1939 &\le \frac{25}{133}(18m_1 + 133m_2 + 392m_3) \le \frac{25}{133}\,1232 < 232, \\
    J &\le \frac{26}{133}(18m_1 + 133m_2 + 392m_3) \le \frac{26}{133}\,1232 < 241.
  \end{aligned}
\end{equation}
The integrality of $R$ and $J$ in~\eqref{eq:b24RKidentities}, together with~\eqref{eq:b24RJupper}, gives $R \le 2170$ and $J \le 240$. Substituting the latter bound into the formula for $K$ in~\eqref{eq:b24RKidentities} yields
\begin{equation}\label{eq:b24RKupperbounds}
  R \le 2170, \qquad J \le 240, \qquad K \ge -624.
\end{equation}

For the lower bound on $R$, \eqref{eq:b24multiplicitybound} and the integrality of the multiplicities imply
\begin{equation}\label{eq:b24smallmultiplicities}
  m_3 \le 3, \qquad m_2 \le 9 - 3m_3.
\end{equation}
If $m_3 \le 2$, combine~\eqref{eq:b24smallmultiplicities} with $J \ge 188$ from~\eqref{eq:24segrebudget} and subtract the expressions for $R$ and $J$ in~\eqref{eq:b24RKidentities}. If $m_3 = 3$, \eqref{eq:b24smallmultiplicities} forces $m_2 = 0$, and we use the expression for $R$ directly. These two cases give
\begin{equation}\label{eq:b24Rlowercases}
  \begin{aligned}
    m_3 \le 2 &: \quad R - 1939 = J - m_2 - 22m_3 \ge 141, \\
    m_3 = 3 &: \quad R - 1939 = 2m_1 + 138 \ge 138.
  \end{aligned}
\end{equation}
Thus~\eqref{eq:b24Rlowercases} implies $R \ge 2077$. Combining this with~\eqref{eq:b24RKupperbounds} and~\eqref{eq:segre6budget} proves~\eqref{eq:projectiveHRrectangle}.
\end{proof}

\section{The Projective Cotangent Bundle and Exclusion of \texorpdfstring{$b_2 = 24$}{b2 = 24}}\label{sec:geometry}

We pass to a very general deformation and construct a K\"ahler class on the projective cotangent bundle. The numerical bounds from the preceding section then produce a negative-definite two-plane, contradicting the Hodge--Riemann relations and excluding $b_2(X) = 24$.

\begin{proposition}\label{prop:nosubvar24}
Let $X$ be an IHS sixfold with $b_2(X) \ge 24$. Then a very general deformation of $X$ contains no positive-dimensional proper compact analytic subvarieties.
\end{proposition}
\begin{proof}
Let $X$ be a very general deformation. By~\cite[Proposition~2.38(i)]{GKLR}, the special Mumford--Tate algebra of $H^*(X,\Q)$ equals the reduced LLV algebra $\so(H^2(X,\Q),q_X)$. Fix a hyperk\"ahler metric on $X$. The Lie algebra of its $SU(2)$ action on cohomology is contained in the reduced LLV algebra~\cite[Section~2.2.2]{GKLR}, and hence, contained in the special Mumford--Tate algebra. The fundamental class of every irreducible compact analytic subvariety $Z \subset X$ is a rational Hodge class, hence is $SU(2)$-invariant. Therefore, $Z$ is trianalytic by~\cite[Theorem~3.5]{VerHypercomplex}.

By~\cite[Theorem~6.2 and its proof]{VerHypercomplex}, the normalization $\widehat{Z}$ is smooth, and the local branches of $Z$ are smooth submanifolds of $X$. Thus, the normalization map $j \colon \widehat{Z} \to X$ is an immersion compatible with the induced complex structures, and the pullback metric is hyperk\"ahler. By the Bogomolov decomposition theorem~\cite[Th\'eor\`eme~1(2)]{Beauville}, $\widehat{Z}$ has a finite \emph{\'etale} cover that is a product of a complex torus and irreducible hyperk\"ahler manifolds.

Fix a positive-dimensional factor $F$ in this cover. The induced map $j_F \colon F \to X$ is a holomorphic immersion. Since the standard representation of $\so(H^2(X,\Q),q_X)$ is irreducible, $H^2(X,\Q)$ has no nontrivial rational Hodge substructures. The kernel of the Hodge morphism $j_F^* \colon H^2(X,\Q) \to H^2(F,\Q)$ is therefore zero or all of $H^2(X,\Q)$. For a K\"ahler class $h$ on $X$, the class $j_F^*h$ is K\"ahler on $F$, so $j_F^*$ is nonzero. Hence, $j_F^*$ is injective, and the same holds over $\C$, preserving Hodge types.

If $0 < \dim_\C Z < 6$, its dimension must be even, so $\dim_\C Z \in \{2, 4\}$. The preceding injection gives $b_2(F) \ge b_2(X) \ge 24$ for each positive-dimensional irreducible hyperk\"ahler factor $F$. However, such a factor has $b_2(F) = 22$ in dimension two, and $b_2(F) \le 23$ in dimension four by Guan~\cite[Theorem~1]{Guan}.

For a positive-dimensional torus factor $T$, the same injection on $H^{1,1}$ gives
\begin{equation}
  22 \le b_2(X)-2 = h^{1,1}(X) \le h^{1,1}(T) = (\dim_\C T)^2 \le 16,
\end{equation}
a contradiction. Hence, no such subvariety exists.
\end{proof}

\begin{remark}
Alternatively, the injection in Proposition~\ref{prop:nosubvar24} follows from orthogonal invariant theory. For a factor immersion $j_F \colon F \to X$ as above, write $\dim_\C F = 2m > 0$. The polarization of the polynomial $\eta \mapsto \int_F (j_F^*\eta)^{2m}$ is a rational Hodge morphism $\Sym^{2m} H^2(X,\Q) \to \Q(-2m)$. By the same Mumford--Tate genericity, this polynomial is orthogonally invariant. Classical invariant theory therefore gives the Fujiki-type relation
\[
  \int_F (j_F^*\eta)^{2m} = c_F q_X(\eta)^m, \qquad \eta \in H^2(X,\C),
\]
with $c_F > 0$ by evaluation on a K\"ahler class. If $j_F^*u = 0$, replacing $\eta$ by $\eta+tu$ and differentiating at $t=0$ gives $q_X(\eta,u) = 0$ whenever $q_X(\eta) \ne 0$. Nondegeneracy of $q_X$ then gives $u = 0$.

Soldatenkov--Verbitsky~\cite[Theorem~1.16 and Corollary~1.17]{SVTrianalytic} proved the corresponding Fujiki-type relation and $H^2$-injectivity for normalizations of positive-dimensional absolutely trianalytic subvarieties. In the very general setting considered here, the required relation follows directly from Mumford--Tate genericity.
\end{remark}

In the range $b_2(X) \ge 24$, we choose $X$ to be simultaneously Mumford--Tate generic and very general in the sense of Anella--H\"oring~\cite[Lemma~3.4]{AHCotangent}: $\Pic(X) = 0$, and every irreducible component of the cycle space of $\mathbb{P}(\Omega_X)$ deforms under small deformations of $X$. These conditions hold on a very general deformation. By Proposition~\ref{prop:nosubvar24} and~\cite[Corollary~3.5]{AHCotangent}, every proper irreducible compact analytic subvariety of $Y = \mathbb{P}(\Omega_X)$ is contained entirely within a fiber of $\pi \colon Y \to X$. We fix such a deformation, noting that the Betti numbers, $R$, and $K$ are deformation invariants.

Let $Q = q_X^\vee \in H^{2,2}(X,\R)$ denote the dual Beauville--Bogomolov--Fujiki class: for any basis $(e_i)$ of $H^2(X,\R)$ with inverse metric matrix $(q_X^{ij})$, $Q = \sum_{i,j} q_X^{ij} e_i e_j$. For any characteristic class $a \in H^{4k}(X,\Q)$ and $N = 3-k$, contracting with the Laplacian $\Delta_{q_X}$ yields
\begin{equation}\label{eq:contraction}
  \int_X a Q^j h^{2N-2j} = C_X(a) \prod_{\ell=0}^{j-1} \frac{b_2 + 2N - 2 - 2\ell}{2N - 1 - 2\ell} \, q_X(h)^{N-j}, \qquad 0 \le j \le N.
\end{equation}

On $Y = \mathbb{P}(\Omega_X)$, write $\zeta = c_1(\mathcal{O}_Y(1))$. For a K\"ahler class $h$ on $X$ with $t = q_X(h) > 0$, set $L = \zeta + h$. Using Lemma~\ref{lem:projectivebundle} and Lemma~\ref{lem:fujiki}, the top intersection of $L$ is given by the volume polynomial
\begin{equation}\label{eq:projectiveHRvolume}
  P(t) := \int_Y L^{11} = 462Rt^3 - 396Rt^2 + (132R - 190080)t + K.
\end{equation}

\begin{proposition}\label{prop:strictrootkahler}
Let $X$ be the fixed very general IHS sixfold above, with $b_2(X) = 24$ and satisfying~\eqref{eq:projectiveHRrectangle}. Let $h$ be a K\"ahler class on $X$ normalized such that $q_X(h) = 3/4$. Then $L = \zeta + h$ is a K\"ahler class on $Y = \mathbb{P}(\Omega_X)$.
\end{proposition}
\begin{proof}
Set $L_s = \zeta + \sqrt{4s/3}\,h$ for $s \ge 3/4$. By the choice of deformation, every proper irreducible compact analytic subvariety $Z \subset Y$ is contained in a fiber, on which $\zeta$ restricts to the Fubini--Study class; hence $\int_Z L_s^{\dim Z} > 0$. For the total space $Y$, evaluating~\eqref{eq:projectiveHRvolume} at $t = 3/4$ gives
\begin{equation}
  P(3/4) = \frac{2277R}{32} - 142560 + K \ge \frac{147441}{32} > 0,
\end{equation}
and its derivative satisfies
\begin{equation}
  P'(3/4) = \frac{2541R}{8} - 190080 \ge \frac{3757017}{8} > 0.
\end{equation}
Furthermore, $P''(s) = R(2772s - 792) > 0$ for all $s \ge 3/4$. Thus, the volume $\int_Y L_s^{11}$ remains strictly positive along the entire ray. The class $\zeta$ is relatively K\"ahler, so $L_s$ is K\"ahler for $s \gg 0$. The ray $(L_s)_{s\ge3/4}$ is connected and numerically positive on every irreducible compact analytic subvariety of $Y$, including $Y$ itself. It therefore lies in the component of the numerically positive cone containing the K\"ahler cone. By Demailly and P\u{a}un~\cite[Theorem~0.1]{DemaillyPaun}, this component is the K\"ahler cone. In particular, $L = L_{3/4}$ is K\"ahler.
\end{proof}

\begin{proposition}\label{prop:projectiveHRsignature}
Let $h$ be a K\"ahler class on $X$, and assume that $L = \zeta + h$ is K\"ahler on $Y$. Then the symmetric bilinear form
\begin{equation}
  B_L(\alpha, \beta) = \int_Y \alpha \cup \beta \cup L^7
\end{equation}
has negative index one on the four-dimensional subspace
\begin{equation}
  W = \operatorname{span}_\R\{\zeta^2, \, \zeta h, \, h^2, \, Q\} \subset H^{2,2}(Y,\R).
\end{equation}
\end{proposition}
\begin{proof}
By Verbitsky~\cite[Theorem~1.5]{Verbitsky}, $\Sym^2 H^2(X,\R)$ injects into $H^4(X,\R)$. The classes $h^2$ and $Q$ are linearly independent because $h^2$ has rank one while $Q$ is non-degenerate; thus $\dim_\R W = 4$.

Let $U = \operatorname{span}_\R\{\zeta, h\} \subset H^{1,1}(Y,\R)$. Hard Lefschetz identifies $LU \subset W$ with $U$. On $LU$, the form $B_L$ coincides with the degree-two intersection form $\int_Y u v L^9$, which has signature $(1,1)$ because $L \in U$. Consequently, we have the orthogonal decomposition
\begin{equation}
  W = LU \mathbin{\perp_{B_L}} P_W, \qquad \text{where } P_W = W \cap (LU)^{\perp_{B_L}}.
\end{equation}
For any $x \in H^{1,1}(X,\R)$ with $q_X(h,x) = 0$, differentiating the contraction formula~\eqref{eq:contraction} in the direction $x$ gives, for $0 \le j < N$,
\[
  \int_X a Q^j x h^{2N-2j-1} = C_X(a) \prod_{\ell=0}^{j-1} \frac{b_2+2N-2-2\ell}{2N-1-2\ell} \, q_X(h)^{N-j-1}q_X(h,x).
\]
After expanding $L^8$ and applying the projective bundle formula, every nonzero term in $\int_Y \alpha x L^8$, for $\alpha \in W$, has this form with $a = 1,c_2$, or $s_4$. Hence $\int_Y \alpha x L^8 = 0$. Since $q_X(h) > 0$, the projective bundle formula gives
\begin{equation}
  H^{1,1}(Y,\R) = U \oplus \pi^*\bigl(h^{\perp_{q_X}} \cap H^{1,1}(X,\R)\bigr).
\end{equation}
Thus, for every $p \in P_W$, orthogonality to $LU$ and the preceding vanishing imply $\int_Y p u L^8 = 0$ for all $u \in H^{1,1}(Y,\R)$. Since $L^8p$ has type $(10,10)$, its pairings with $H^{2,0}(Y)$ and $H^{0,2}(Y)$ also vanish. Poincar\'e duality then gives $L^8p = 0$. Therefore, every class in $P_W$ is primitive with respect to the Lefschetz operator $L$. By the classical Hodge--Riemann bilinear relations, $B_L$ is positive definite on primitive real $(2,2)$-classes. Hence, $P_W$ is positive definite, and $B_L|_W$ has signature $(3,1)$, which is of negative index one.
\end{proof}

We now apply these results to exclude $b_2(X) = 24$. At this value, the contraction identity~\eqref{eq:contraction} evaluates to:
\begin{equation}
  \begin{gathered}
    \int_X Q h^4 = \frac{28R}{5} t^2, \qquad \int_X Q^2 h^2 = \frac{728R}{15} t, \\
    \int_X c_2 Q h^2 = \frac{52R}{5} t, \qquad \int_X c_2 Q^2 = \frac{1248R}{5}, \qquad \int_X s_4 Q = 24\left(\frac{12R}{5} - 3456\right).
  \end{gathered}
\end{equation}

\begin{proof}[Proof of Theorem~\ref{thm:main}]
By Proposition~\ref{prop:26} and Lemma~\ref{prop:24segre}, we have $b_2(X) \le 24$. To rule out $b_2(X) = 24$, pass to the very general deformation chosen above and normalize the K\"ahler class $h$ such that $t = q_X(h) = 3/4$. By Proposition~\ref{prop:strictrootkahler}, $L = \zeta + h$ is K\"ahler, and by Proposition~\ref{prop:projectiveHRsignature}, $B_L$ has negative index one on $W$.

Consider the explicit two-dimensional subspace of $W$ spanned by
\begin{equation}
  \alpha = \zeta^2 + \zeta h, \qquad \beta = -4\zeta^2 + 6\zeta h - Q.
\end{equation}
Expanding the intersections $\int_Y \alpha^2 L^7$, $\int_Y \alpha \beta L^7$, and $\int_Y \beta^2 L^7$ using Lemma~\ref{lem:projectivebundle} and the contraction formulas at $t = 3/4$, the Gram matrix of $B_L$ on $\operatorname{span}_\R\{\alpha, \beta\}$ is
\begin{equation}\label{eq:sixfoldFixedNegativePlane}
  M = M_0 + K \begin{pmatrix} 1 & -4 \\ -4 & 16 \end{pmatrix},
\end{equation}
where
\begin{equation}
  M_0 = \begin{pmatrix}
    \frac{243R}{16} - 93312 & -\frac{591R}{5} + 248832 \\
    -\frac{591R}{5} + 248832 & \frac{27477R}{80} - 756864
  \end{pmatrix}.
\end{equation}
On the range $2077 \le R \le 2170$ established in~\eqref{eq:projectiveHRrectangle}, we compute:
\begin{equation}
  (M_0)_{11} \le -\frac{482841}{8}, \qquad (M_0)_{22} \le -\frac{92403}{8}, \qquad |(M_0)_{12}| \le 7662.
\end{equation}
Because both diagonal entries are strictly negative and satisfy $(M_0)_{ii} < -|(M_0)_{12}|$, the matrix $M_0$ is strictly negative definite. Furthermore, the second matrix in~\eqref{eq:sixfoldFixedNegativePlane} is $K v v^t$ for $v = (1, -4)^t$. Since $K \le 0$, this matrix is negative semidefinite.

Consequently, the Gram matrix $M$ is strictly negative definite, meaning that $B_L$ is negative definite on the two-plane $\operatorname{span}_\R\{\alpha, \beta\} \subset W$. This forces $B_L|_W$ to have negative index at least two, contradicting Proposition~\ref{prop:projectiveHRsignature}. Thus, $b_2(X) = 24$ is impossible, completing the proof of Theorem~\ref{thm:main}.
\end{proof}

\section{Odd Betti Numbers and Second Chern Class Purity}\label{sec:endpoint}

We bound $b_5(X)$ for $19 \le b_2(X) \le 23$ and deduce the vanishing of all odd Betti numbers when $b_2(X) = 23$. Under the purity hypothesis $c_2(X) \in \SH^4(X,\Q)$, we then determine the possible values of $b_2(X)$ and the full Betti vector at $b_2(X) = 23$.

\begin{corollary}\label{cor:fifthbetti}
Let $X$ be an IHS sixfold with $19 \le b_2(X) \le 23$. Then $b_5(X)$ belongs to the following sets:
\begin{equation}
  b_5(X) \in
  \begin{cases}
    \{0, 1024, 2048\}, & \text{if } b_2(X) = 19, \\
    \{0, 1024\},       & \text{if } 20 \le b_2(X) \le 22, \\
    \{0\},             & \text{if } b_2(X) = 23.
  \end{cases}
\end{equation}
\end{corollary}
\begin{proof}
By Lemma~\ref{prop:highbodd}, $b_3(X) = 0$, which implies $b_5 = d_{b_2} v$. For $b_2 = 19$, inequality~\eqref{eq:odd19budget} forces $v \le 4$. Because $v$ must be even by Lemma~\ref{prop:oddgrowth}, we have $v \in \{0, 2, 4\}$, which proves the first case.

For $20 \le b_2 \le 23$, define the constants
\begin{equation}
  \eta_{b_2} = \frac{4(b_2-11)}{b_2^2-13b_2+2}, \qquad \rho_{b_2} = \frac{(b_2-11)(b_2+1)(b_2+2)(b_2-1)}{11(b_2^2-13b_2+2)}.
\end{equation}
Setting $x = k(b_2+2-k)$, the binomial coefficients satisfy
\begin{equation}
  \frac{\eta_{b_2}\operatorname{Sal}(E_k) - \mathcal{R}(E_k) - h^{3,1}(E_k)}{b_4(E_k)} = \frac{22(x-2b_2)(x-\rho_{b_2})}{b_2(b_2-1)x} \ge 0,
\end{equation}
for all $1 \le k \le \left\lfloor \frac{(b_2+2)}{2} \right\rfloor$. Combining~\eqref{eq:Salamonmodules} and~\eqref{eq:oddSegrebudget} yields
\begin{equation}
  \frac{b_5}{2} \le b_2^2 - 57b_2 + 1638 + \eta_{b_2} f(b_2).
\end{equation}
Multiplying this bound by two and evaluating it for $b_2 \in \{20, 21, 22, 23\}$ yields upper bounds on $b_5$ of $\frac{141556}{71}, \frac{32508}{17}, \frac{45864}{25}$, and $1760$, all strictly smaller than $2048$. Divisibility by $d_{b_2}$ together with the parity of $v$ establishes the claim.
\end{proof}

\begin{proof}[Proof of Theorem~\ref{thm:oddendpoint}]
Part~(i) is established by Lemma~\ref{prop:highbodd}. For part~(ii), Corollary~\ref{cor:fifthbetti} shows that $b_5(X) = 0$ when $b_2(X) = 23$. By simple connectedness and Poincar\'e duality, all odd Betti numbers must vanish.
\end{proof}

\begin{proof}[Proof of Theorem~\ref{thm:pure}]
Assume $c_2(X) \in \SH^4(X,\Q)$, and set $(a,d,e) = \int_X (c_2^3, c_2 c_4, c_6)$. By Beckmann--Song~\cite[Proposition~2.10]{BS}, $[\sqrt{\td_X}]_8 \in \SH^8(X,\Q)$, which forces $c_4 \in \SH^8(X,\Q)$. By Looijenga--Lunts~\cite[Propositions~4.5(ii) and~4.9]{LL}, the reduced LLV algebra $\so(H^2, q_X)$ acts by derivations and fixes the Pontryagin classes. Since $p_1 = -2c_2$ and $p_2 = c_2^2 + 2c_4$, it also fixes $c_2$ and $c_4$.

The orthogonal and special orthogonal groups have the same polynomial invariants on one vector. By the first fundamental theorem~\cite[Theorem~5.2.2(1)]{GoodmanWallach}, the invariant lines in $\Sym^2 H^2(X,\C)$ and $\Sym^4 H^2(X,\C)$ are generated $Q$ and $Q^2$, where $Q$ is the inverse of quadratic tensor $q_X$. Since the reduced LLV algebra is semisimple, complete reducibility~\cite[Theorem~3.3.12]{GoodmanWallach} allows invariant classes in the equivariant cup-product quotients $\SH^4$ and $\SH^8$ to lift to invariant tensors. Hence, $c_2$ is a scalar multiple of $Q$ and $c_4$ is a scalar multiple of $Q^2$.

Equality in Beckmann--Song's inequality gives $R = \frac{2880(b_2+4)}{b_2+1}$. Comparing the contraction formula~\eqref{eq:contraction} with the generalized Fujiki constants from Lemma~\ref{lem:fujiki}, we obtain
\[
  c_2 = \frac{6}{b_2+4}Q, \qquad c_4 = \frac{18(b_2-3)}{(b_2+4)^2(b_2+2)}Q^2,
\]
and
\[
  \int_X Q^3 = \frac{R\,b_2(b_2+2)(b_2+4)}{15}.
\]
Substituting the value of $R$ gives
\begin{equation}\label{eq:pureChern}
  a = \frac{41472b_2(b_2+2)}{(b_2+1)(b_2+4)}, \qquad d = \frac{20736b_2(b_2-3)}{(b_2+1)(b_2+4)}.
\end{equation}
Using the relation $a - d = 23616 + 24\chi^1$ from Lemma~\ref{lem:chernchi}, we deduce
\begin{equation}
  \chi^1 = -120 - \frac{1728}{b_2+1} + \frac{3456}{b_2+4}.
\end{equation}
Multiplying by $(b_2+1)(b_2+4)$ and reducing modulo $(b_2+1)$ forces $(b_2+1) \mid 5184$. Similarly, the integrality of $a$ and $d$ modulo $(b_2+4)$ implies $(b_2+4) \mid \gcd(331776, 580608) = 82944$. By Theorem~\ref{thm:main}, $4 \le b_2+1 \le 24$. The only integers in this range satisfying both divisibility conditions are $b_2+1 \in \{6, 9, 24\}$, which proves that $b_2(X) \in \{5, 8, 23\}$.

Finally, assume $b_2(X) = 23$. Theorem~\ref{thm:oddendpoint} implies that all odd Betti numbers vanish, and equality in Beckmann--Song's inequality gives $R = 3240$. The representation-theoretic identities~\eqref{eq:Smodules} and~\eqref{eq:Salamonmodules} become
\begin{equation}
  \sum_{i \in \mathcal{I}} m_i\operatorname{Sal}(E_i) + \tau = 116, \qquad \sum_{i \in \mathcal{I}} m_i\mathcal{R}(E_i) = 23.
\end{equation}
For $x = k(25-k)$, the binomial ratios~\eqref{eq:exteriorratios} give
\[
  \frac{(23/116)\operatorname{Sal}(E_k)-\mathcal{R}(E_k)}{b_4(E_k)} = \frac{23(x-46)(29x-1650)}{14674x}.
\]
The allowed values of $x$ begin with $24,46,66$ and then increase. Thus $\mathcal{R}(E_k) \le \frac{23}{116}\operatorname{Sal}(E_k)$ for all exterior summands, with equality if and only if $k = 2$. Consequently,
\begin{equation}
  0 = \frac{23}{116} \tau + \sum_{i \in \mathcal{I}} m_i \left(\frac{23}{116}\operatorname{Sal}(E_i) - \mathcal{R}(E_i)\right).
\end{equation}
Since every term on the right-hand side is non-negative, we must have $\tau = 0$ and $m_i = 0$ for all $i \ne 2$. Since $\operatorname{Sal}(E_2) = 116$, we have $m_2 = 1$. Thus, as an LLV module,
\begin{equation}
  H^{\mathrm{even}}(X,\C) \simeq \SH \oplus E_2\ , \ \ E_2= \Lambda^2 \C^{25}.
\end{equation}
In particular, $b_4 = \binom{24}{2} + 23 = 299$. Salamon's relation~\eqref{eq:intro-salamon} then computes $b_6 = 6(299) + 30(23) + 70 = 2554$. Poincar\'e duality determines the remainder of the Betti vector.
\end{proof}

\begin{remark}
At $b_2(X) = 23$, the purity hypothesis in Theorem~\ref{thm:pure} determines the Betti vector and LLV decomposition; they coincide with those of $\mathrm{K3}^{[3]}$.
\end{remark}


\renewcommand{\bibliofont}{\small}

\end{document}